\documentclass[a4paper,12pt]{article}
\usepackage{ucs}
\usepackage{amsfonts, amssymb, amsmath, amsthm, amscd}
\usepackage{graphicx}
\usepackage{cite}
\ifx\undefined \pdfgentounicode \else
\input{glyphtounicode} \pdfgentounicode=1
\fi

\author{A.A. Vasil'eva}
\title{Estimates for the Gelfand widths of intersections of finite-dimensional balls}
\date{}
\begin{document}

\maketitle

\newenvironment{Biblio}{%
                  \renewcommand{\refname}{\footnotesize REFERENCES}%
                  \begin{thebibliography}{99}%
                  \itemsep -0.2em}%
                  {\end{thebibliography}}

\def\inff{\mathop{\smash\inf\vphantom\sup}}
\renewcommand{\le}{\leqslant}
\renewcommand{\ge}{\geqslant}
\newcommand{\sgn}{\mathrm {sgn}\,}
\newcommand{\inter}{\mathrm {int}\,}
\newcommand{\dist}{\mathrm {dist}}
\newcommand{\supp}{\mathrm {supp}\,}
\newcommand{\R}{\mathbb{R}}
\newcommand{\Z}{\mathbb{Z}}
\newcommand{\N}{\mathbb{N}}
\newcommand{\Q}{\mathbb{Q}}
\theoremstyle{plain}
\newtheorem{Trm}{Theorem}
\newtheorem{trma}{Theorem}

\newtheorem{Def}{Definition}
\newtheorem{Cor}{Corollary}
\newtheorem{Lem}{Lemma}
\newtheorem{Rem}{Remark}
\newtheorem{Sta}{Proposition}

\renewcommand{\proofname}{\bf Proof}
\renewcommand{\thetrma}{\Alph{trma}}

\begin{abstract}
In this paper, we obtain order estimates for the Gelfand widths of intersections of finite-dimensional balls under some conditions on parameters.
\end{abstract}

\section{Introduction}

In this paper, we consider the problem of estimating the Gelfand $n$-widths of an intersection of finite-dimensional balls for $n\le aN$, where $a\in (0, \, 1/2]$ is an absolute constant. The similar problem for the Kolmogorov widths was considered in \cite{galeev1, vas_ball_inters}. For some conditions on the parameters, we obtain the order estimates for the Gelfand widths; as a consequence, under some conditions on the parameters, we get the order estimates for linear widths.

Let us give the definitions of the Gelfand, Kolmogorov and linear widths.

Let $X$ be a linear normed space. By $X^*$ we denote the dual space for $X$, by ${\cal L}_n(X)$, the family of all linear subspaces in $X$ of dimension at most $n$, by $L(X, \, X)$, the space of linear continuous operators on $X$. For $A\in L(X, \, X)$, by ${\rm rk}\, A$ we denote the dimension of the range of $A$. Let $M\subset X$ be a nonempty set, $n\in \Z_+$. The Gelfand, Kolmogorov and linear $n$-widths of the set $M$ in $X$ are defined, correspondingly, as follows: 
$$
d^n(M, \, X) = \inf _{x_1^*, \, \dots, \, x_n^*\in X^*} \sup \Bigl\{\|x\|:\; x\in M\cap \Bigl(\cap _{j=1}^n \ker x_j^*\Bigr)\Bigr\};
$$
$$
d_n(M, \, X) = \inf _{L\in {\cal L}_n(X)} \sup _{x\in M} \inf
_{y\in L} \|x-y\|;
$$
$$
\lambda_n(M, \, X) = \inf _{A\in L(X, \, X), \, {\rm rk}\, A\le n} \sup _{x\in M} \|x-Ax\|.
$$

From the definitions it follows that $d_n(M, \, X)\le \lambda_n(M, \, X)$, $d^n(M, \, X) \le \lambda_n(M, \, X)$.

Let $N\in \N$, $1\le p\le \infty$, $(x_i)_{i=1}^N\in \mathbb{R}^N$. We set $\|(x_i)_{i=1}^N\|_{l_p^N} = \left(\sum \limits _{i=1}^N |x_i|^p\right)^{1/p}$ for $p<\infty$, $\|(x_i)_{i=1}^N\|_{l_p^N} = \max _{1\le i\le N}|x_i|$ by $p=\infty$. The space $\mathbb{R}^N$ with this norm is denoted by $l_p^N$. By $B_p^N$ we denote the unit ball in $l_p^N$.

Given $1\le p\le \infty$, we denote by $p'$ the dual exponent of $p$ defined by the equation $\frac{1}{p}+ \frac{1}{p'}=1$.

The problem of estimating $d_n(B_p^N, \, l_q^N)$, $d^n(B_p^N, \, l_q^N)$ and $\lambda_n(B_p^N, \, l_q^N)$ was studied in \cite{k_p_s, stech_poper, pietsch1, stesin, kashin_oct, bib_kashin, kashin_matr, gluskin1, bib_gluskin, garn_glus}. For $p\ge q$, the sharp values of these widths are known \cite{pietsch1, stesin}, as well as the sharp values of $d_n(B_1^N, \, l_2^N)= \lambda_n(B_1^N, \, l_2^N)$ \cite{k_p_s, stech_poper}. Hence from the duality formulas $d_n(B_p^N, \, l_q^N)=d^n(B_{q'}^N, \, l_{p'}^N)$, $\lambda_n(B_p^N, \, l_q^N)=\lambda_n(B_{q'}^N, \, l_{p'}^N)$ (see \cite{ioffe_tikh}) we get the sharp values of $d^n(B_2^N, \, l_\infty^N)=\lambda_n(B_2^N, \, l_\infty^N)$. Order estimates for $d_n(B_p^N, \, l_q^N)$, $d^n(B_p^N, \, l_q^N)$ and $\lambda_n(B_p^N, \, l_q^N)$ are obtained, correspondingly, for $\{1\le p\le q\le \infty:\; q<\infty\; \text{or }(q=\infty,\; p\ge 2)\}$, $\{1\le p\le q\le \infty:\; p>1\, \text{or }(p=1,\; q\le 2)\}$, $\{1\le p\le q\le \infty:\; (p, \, q)\ne (1, \, \infty)\}$ (see \cite{bib_kashin, kashin_matr, gluskin1, bib_gluskin, garn_glus}).

In \cite{galeev1} order estimates for $d_n(\cap _{\alpha\in A}\nu_\alpha B^N_{p_\alpha}, \, l_q^N)$ were obtained for $N=2n$ (here $\nu_\alpha>0$, $\alpha\in A$). In \cite{vas_ball_inters}, this result was generalized for the case $n\le N/2$; after that, the similar problem for mixed norms was studied \cite{vas_mix_sev}.

The order estimates of $d_n(\cap _{\alpha\in A}\nu_\alpha B^N_{p_\alpha}, \, l_q^N)$ (where $n\le N/2$) were obtained for $1\le q<\infty$ and arbitrary $p_\alpha$, as well as for $q=\infty$ and $p_\alpha\ge 2$, $\alpha\in A$. Also some simple cases were mentioned, when $\lambda_n(\cap _{\alpha\in A}\nu_\alpha B^N_{p_\alpha}, \, l_q^N)$ and $d^n(\cap _{\alpha\in A}\nu_\alpha B^N_{p_\alpha}, \, l_q^N)$ can be similarly estimated \cite[Remark 1]{vas_ball_inters}. In particular, for the Gelfand widths the order estimates can be easily written if 1) $p_\alpha\ge q$, $\alpha\in A$, 2) $2\le p_\alpha\le q$, $\alpha\in A$; they are given, correspondingly, by formulas $\inf _{\alpha\in A} \nu_\alpha N^{1/q-1/p_\alpha}$ and $\inf _{\alpha\in A} \nu_\alpha$.

In this paper we consider two more areas of parameters, where we could obtain order estimates of $d^n(\cap _{\alpha\in A}\nu_\alpha B^N_{p_\alpha}, \, l_q^N)$.

First we consider the case of an intersection of a finite family of balls $\cap _{j=1}^r \nu_j B^N_{p_j}$, when 
\begin{align}
\label{pj_upor}
p_1< p_2< \dots < p_r, 
\end{align}
\begin{align}
\label{ne_vkl}
\nu_1\ge \nu_2\ge \dots \ge \nu_r, \quad \nu_1 N^{-1/p_1}\le \nu_2 N^{-1/p_2} \le \dots \le \nu_r N^{-1/p_r}.
\end{align}
The condition \eqref{pj_upor} can be satisfied by choosing the numeration; if $p_i=p_j$ for some $i\ne j$, the problem can be reduced to the case of less number of balls. If \eqref{ne_vkl} fails, we have $\nu_iB^N_{p_i} \subset \nu_j B^N_{p_j}$ or $\nu_iB^N_{p_i} \supset \nu_j B^N_{p_j}$ for some $i\ne j$, and the problem can be reduced to the case of less number of balls. Hence in the case of an intersection of a finite number of balls it suffices to obtain the estimates under the conditions \eqref{pj_upor}, \eqref{ne_vkl}.

Let $X$, $Y$ be sets, $f_1$, $f_2:\ X\times Y\rightarrow \mathbb{R}_+$.
We write $f_1(x, \, y)\underset{y}{\lesssim} f_2(x, \, y)$ (or
$f_2(x, \, y)\underset{y}{\gtrsim} f_1(x, \, y)$) if for each
$y\in Y$ there is $c(y)>0$ such that $f_1(x, \, y)\le
c(y)f_2(x, \, y)$ for all $x\in X$; $f_1(x, \,
y)\underset{y}{\asymp} f_2(x, \, y)$ if $f_1(x, \, y)
\underset{y}{\lesssim} f_2(x, \, y)$ and $f_2(x, \,
y)\underset{y}{\lesssim} f_1(x, \, y)$.

\begin{Trm}
\label{main1} Let $2\le q\le \infty$, $1<p_1< p_2<\dots < p_r\le q$. Suppose that \eqref{ne_vkl} holds and $p_1<2$. Then for $n\le N/2$
$$
d^n(\cap _{j=1}^r \nu_j B_{p_j}^N, \, l_q^N) \underset{p_1}{\asymp} \min \{\nu_1 n^{-1/2}N^{1/p_1'}, \, \nu_r\}.
$$
\end{Trm}

\begin{Trm}
\label{main2} Let $2\le q\le \infty$, $2< p_1< p_2< \dots < p_r$. Suppose that
\eqref{ne_vkl} holds and $p_1<q<p_r$. Given $p_i\le q$, $p_j\ge q$ we define the numbers $\lambda_{ij}\in [0, \, 1]$ by the equation
\begin{align}
\label{lambda_def} \frac 1q = \frac{1-\lambda_{ij}}{p_i} + \frac{\lambda_{ij}}{p_j}.
\end{align}
Then for $n\le N/4$
$$
d^n(\cap _{j=1}^r \nu_j B_{p_j}^N, \, l_q^N) \asymp \min _{p_i\le q, \, p_j\ge q} \nu_i^{1-\lambda_{ij}} \nu_j^{\lambda_{ij}}.
$$
\end{Trm}

Now we formulate the generalization of Theorems \ref{main1} and \ref{main2} for an intersection of an arbitrary family of balls; in addition, we write the estimates for linear widths. Here \eqref{pj_upor}, \eqref{ne_vkl} may fail.
\begin{Trm}
\label{cor1} Let $N\in \N$, $n\in \Z_+$, $2\le q\le \infty$. Let $A$ be a nonempty set, and let $\nu_\alpha>0$, $1<p_\alpha\le \infty$, $\alpha\in A$. Suppose that $\hat p:=\inf _{\alpha\in A} p_\alpha >1$.
\begin{enumerate}
\item If $p_\alpha\le q$ for each $\alpha\in A$, then, for $n\le N/2$,
$$
d^n(\cap _{\alpha\in A} \nu_\alpha B_{p_\alpha}^N, \, l_q^N) \underset{\hat{p}}{\asymp} 
\inf _{\alpha\in A} (\nu_\alpha\min \{1, \, n^{-1/2}N^{1/p_\alpha'}\}). 
$$
If, in addition, $1/q+1/p_\alpha\le 1$ for all $\alpha\in A$, then the similar order estimate holds for linear widths. 
\item Let $p_\alpha\ge 2$ for each $\alpha\in A$. Given $p_\alpha<q$, $p_\beta>q$, we define the numbers $\lambda_{\alpha \beta}$ by the equations $\frac 1q = \frac{1-\lambda_{\alpha \beta}}{p_\alpha} + \frac{\lambda_{\alpha \beta}}{p_\beta}$. Then, for $n\le N/4$,
$$
d^n(\cap _{\alpha\in A} \nu_\alpha B_{p_\alpha}^N, \, l_q^N) \asymp \lambda_n(\cap _{\alpha\in A} \nu_\alpha B_{p_\alpha}^N, \, l_q^N)\asymp
$$
$$
\asymp \min \{\inf _{p_\alpha\ge q} \nu_\alpha N^{1/q-1/p_\alpha}, \, \inf _{p_\alpha\le q} \nu_\alpha, \, \inf _{p_\alpha<q, \, p_\beta>q} \nu_\alpha^{1-\lambda_{\alpha\beta}} \nu_\beta^{\lambda_{\alpha \beta}}\}.
$$
\end{enumerate}
\end{Trm}

We also obtain order estimates for the Gelfand widths of an intersection of two balls $\nu_1B_{p_1}^N\cap \nu_2 B_{p_2}^N$ in $l_2^N$ for $1<p_1<2<p_2$, $1\le \frac{\nu_1}{\nu_2} \le N^{1/p_1-1/2}$. In the case $N^{1/p_1-1/2}< \frac{\nu_1}{\nu_2} \le N^{1/p_1-1/p_2}$, the order estimate is obtained only for small $n$.

\begin{Trm}
\label{main3} Let $1<p_1<2<p_2\le \infty$, $1\le \frac{\nu_1}{\nu_2} \le N^{1/p_1-1/p_2}$. We define the number $\lambda\in (0, \, 1)$ by the equation
\begin{align}
\label{121lp1}
\frac 12 = \frac{1-\lambda}{p_1} + \frac{\lambda}{p_2}.
\end{align}
Then there is an absolute constant $a_0>0$ such that
\begin{enumerate}
\item for $\frac{\nu_1}{\nu_2} \le N^{1/p_1-1/2}$, $n\le a_0N$, we have
$$
d^n(\nu_1B_{p_1}^N\cap \nu_2 B_{p_2}^N, \, l_2^N) \underset{p_1}{\asymp} \min \{\nu_1^{1-\lambda} \nu_2^\lambda, \, \nu_1 n^{-1/2}N^{1/p_1'}\};
$$
\item for $\frac{\nu_1}{\nu_2} > N^{1/p_1-1/2}$, $n\le a_0\left(\frac{\nu_1}{\nu_2}\right)^{2\lambda-2}N$, we have
$$
d^n(\nu_1B_{p_1}^N\cap \nu_2 B_{p_2}^N, \, l_2^N) \asymp \nu_1^{1-\lambda} \nu_2^\lambda.
$$
\end{enumerate}
\end{Trm}

The paper is organized as follows. In \S 2 some well-known results are formulated. In \S 3 we generalize Gluskin's method \cite{gluskin1}; it will be used in estimating the widths from below. In \S 4, the main results are proved. In \S 5, we apply these results in estimating the Gelfand widths of an intersection of Sobolev classes on a John domain. Notice that the problem of estimating the Kolmogorov widths of an intersection of Sobolev classes was studied in \cite{galeev1, galeev2, vas_int_sob}. In \S 6, we obtain some generalizations of Lemma \ref{low_est_lem}.

\section{Preliminary results}

Let us formulate Ioffe's and Tikhomirov's result about the duality of the Kolmogorov and Gelfand widths (the finite-dimensional case).

\begin{trma}
\label{dual_teor} {\rm (see \cite{ioffe_tikh})}. Let $X=(\R^N, \, \|\cdot\|_X)$, $Y=(\R^N, \, \|\cdot\|_Y)$, let $X^*$, $Y^*$ be duals of $X$ and $Y$, and let $B_X$, $B_Y$, $B_{X^*}$, $B_{Y^*}$ be the corresponding unit balls. Then, for all $n=0, \, 1, \, \dots, \, N$,
$$
d_n(B_X, \, Y) = d^n(B_{Y^*}, \, X^*).
$$
\end{trma}

We write the well-known order estimates for the Gelfand and linear widths of the finite-dimensional balls in the cases which will be used later.

\begin{trma}
\label{spg_teor}
{\rm (see \cite{pietsch1, stesin, bib_gluskin}).} Let $n\le N/2$, $1\le q\le \infty$, $1<p\le \infty$.
\begin{enumerate}
\item If $p\ge q$, then $d^n(B_p^N, \, l_q^N) = \lambda_n(B_p^N, \, l_q^N) \asymp N^{1/q-1/p}$.

\item If $2\le p\le q$, then $d^n(B_p^N, \, l_q^N) \asymp \lambda_n(B_p^N, \, l_q^N) \asymp 1$.

\item If $1<\hat p\le p\le 2\le q$, then $d^n(B_p^N, \, l_q^N) \underset{\hat p}{\asymp} \min \{1, \, n^{-1/2}N^{1/p'}\}$; if, in addition, $\frac{1}{p}+\frac{1}{q}\le 1$, then $\lambda_n(B_p^N, \, l_q^N) \underset{\hat p}{\asymp} \min \{1, \, n^{-1/2}N^{1/p'}\}$.
\end{enumerate}
\end{trma}

We need a remark about assertion 3 of Theorem \ref{spg_teor}. In {\rm \cite{bib_gluskin}} the constants in the order equality depended on $p$; here they depend on $\hat p$. From the trivial estimate $d^n(B_p^N, \, l_q^N)\le 1$ and the inclusion $B_p^N \subset N^{1/\hat p-1/p}B_{\hat p}^N$ it follows that the constant in the upper estimate of the Gelfand width for $p\in [\hat p, \, 2]$ depends only on $\hat p$. In the lower estimate the constant is absolute. In order to prove this, we slightly modify the arguments from \cite{gluskin1}. To this end, we use Lemma \ref{low_est_lem} (see \S 3 in this paper) for $r=1$, $\nu_1=1$, $s=1$ together with the inequality $\|x\|_{l_2^N}\le N^{1/p-1/2}\|x\|_{l_{p'}^N}$; this yields the estimate of $d^n(B_p^N, \, l_q^N)$ for $n\le a\cdot N^{2/p'}$, where $a$ is an absolute constant. The estimate in the case  $a\cdot N^{2/p'}\le n\le N/2$ can be derived from the estimate in the case $n\le a\cdot N^{2/p'}$; the arguments are similar as in \cite{gluskin1}. This together with the inequality $d^n(M, \, X)\le \lambda_n(M, \, X)$ implies that the constant in the lower estimates of the linear widths is also absolute. Finally, from the proof of the upper estimates for the linear widths in \cite{bib_gluskin} we see that the constant in the order inequality can be chosen continuously depending on $p$; taking the maximum of these vaues over $p\in [\hat p, \, 2]$, we get the constant depending only on $\hat p$.

The following assertion is a particular case of Galeev's result \cite[Theorem 2]{galeev1}; it also follows from H\"{o}lder's inequality.

\begin{trma}
\label{incl_teor} {{\rm (see \cite{galeev1}).}} Let $\nu_1, \, \nu_2>0$, $1\le p_1, \, p_2\le \infty$, $\lambda\in [0, \, 1]$, $\frac{1}{q} = \frac{1-\lambda}{p_1} + \frac{\lambda}{p_2}$. Then
$$
\nu_1B^N_{p_1} \cap \nu_2 B^N_{p_2} \subset \nu_1^{1-\lambda} \nu_2^\lambda B_q^N.
$$
\end{trma}

Let $m, \, k\in \N$, $1\le p<\infty$, $1\le \theta<\infty$. By $l_{p,\theta}^{m,k}$ we denote the space $\R^{mk}$ with the norm
$$
\|(x_{i,j})_{1\le i\le m, \, 1\le j\le k}\|_{l_{p,\theta}^{m,k}} = \left(\sum \limits _{j=1}^k\left(\sum \limits _{i=1}^m |x_{i,j}|^p\right)^{\theta/p}\right)^{1/\theta}.
$$
For $p=\infty$ or $\theta=\infty$, the definition is naturally modified. By $B_{p,\theta}^{m,k}$ we denote the unit ball of the space $l_{p,\theta}^{m,k}$.

The following result was proved by Malykhin and Rjutin \cite{mal_rjut}.
\begin{trma}
{\rm (see \cite{mal_rjut}).} Let $m$, $k\in \N$, $n\le mk/2$. Then
\begin{align}
\label{mr_th}
d^n(B^{m,k}_{2,\infty}, \, l^{m,k}_{\infty,1}) \asymp k.
\end{align}
\end{trma}

\section{Auxiliary assertions}

In order to prove Theorems \ref{main1}, \ref{main3}, we will use some generalizations of the method from \cite{gluskin1}.

Let $(X, \, \|\cdot \|)$ be a normed space. Given $x\in X$, we denote by $f_x$ a supporting functional in $x$; i.e., $f_x\in X^*$, $\|f_x\|_{X^*}=1$, $f_x(x) = \|x\|$.

\begin{Lem}
\label{qnorm}
Let $(X, \, \|\cdot \|)$ be a normed space. Then there is an absolute constant $c>0$ such that, for all $x, \, h\in X$,
\begin{align}
\label{quad_norm} \|x+h\|^2 \ge \frac{\|x\|^2}{2} + 2\|x\|\cdot f_x(h) + c\|h\|^2.
\end{align}
\end{Lem}
\begin{proof}
If $x=0$, then the inequality holds with $c=1$. Further we consider only $x\ne 0$.

First we prove that
\begin{align}
\label{quad_subd} \|x+h\|^2 \ge \|x\|^2 + 2\|x\|\cdot f_x(h).
\end{align}

It is well-known that $f_x$ is a subgradient of the norm at the point $x$; i.e., for each $h\in X$ we have $\|x+h\| \ge \|x\| + f_x(h)$. If $\|h\|\le \|x\|$, the right-hand side of the inequality is non-negative; hence $\|x+h\|^2 \ge (\|x\| + f_x(h))^2 \ge \|x\|^2 + 2\|x\| f_x(h)$. Therefore, \eqref{quad_subd} holds for $\|h\|\le \|x\|$. Since the left-hand side of \eqref{quad_subd} in convex in $h$ and the right-hand side is affine, this inequality holds for all $h\in X$.

From \eqref{quad_subd} it follows that if $\|h\| \le t\|x\|$ with some $t>0$, then
\begin{align}
\label{h_small} \|x+h\|^2 \ge \frac{\|x\|^2}{2} + 2\|x\|\cdot f_x(h) + \frac{\|h\|^2}{2t^2}.
\end{align}

We prove that there is an absolute constant $t_0>0$ such that, for $\|h\|\ge t_0\|x\|$,
\begin{align}
\label{h_big} \|x+h\|^2 \ge \frac{\|x\|^2}{2} + 2\|x\|\cdot f_x(h) + \frac{\|h\|^2}{4}.
\end{align}
Indeed, from the inequality $(a+b)^2\le 2a^2 + 2b^2$ it follows that
\begin{align}
\label{xh}
\|x+h\|^2+\|x\|^2 \ge \frac{\|h\|^2}{2}.
\end{align}
If
\begin{align}
\label{xh1}
\frac{\|h\|^2}{4} \ge 2\|x\|^2 +2\|x\| f_x(h),
\end{align}
then this together with \eqref{xh} implies \eqref{h_big}. We show that if $\|h\| = t\|x\|$, then for sufficiently large $t>0$ the inequality \eqref{xh1} holds. Indeed, since $\|f_x\|_{X^*}=1$, it suffices to check that $\frac{\|h\|^2}{4} \ge 2\|x\|^2 +2\|x\| \cdot \|h\|$, or $\frac{t^2}{4}\ge 2 + 2t$. It holds for $t\ge 10$.

From \eqref{h_small} and \eqref{h_big} we get \eqref{quad_norm}.
\end{proof}

Let $G=S_N \times \{-1, \, 1\}^N$, where $S_N$ is the group of permutations of $N$ elements. For $x=(x_1, \, \dots, \, x_N)\in \R^N$, $g=(\sigma, \, \varepsilon_1, \, \dots, \, \varepsilon_N)$ we denote
\begin{align}
\label{dx_dejstv}
g(x) = (\varepsilon_1 x_{\sigma(1)}, \, \dots, \, \varepsilon_N x_{\sigma(N)}).
\end{align}

Let $s\in \{1, \, \dots, \, N\}$, $1\le q\le \infty$. We set $\hat x = (\hat x_1, \, \dots, \, \hat x_N)$, where
\begin{align}
\label{hat_xj_def}
\hat x_j = \begin{cases} s^{-1/q'}, & 1\le j\le s, \\ 0, & s+1\le j\le N.\end{cases}
\end{align}

\begin{Lem}
\label{low_est_lem} Let $s\in \{1, \, \dots, \, N\}$, $1\le q\le \infty$, let $c$ be the absolute constant from Lemma {\rm \ref{qnorm}}, and let $\nu_j>0$, $1\le p_j\le \infty$, $1\le j\le r$. Let $X$ be the space $\R^N$ with the norm
\begin{align}
\label{norm_x} \|x\|_X = \max _{1\le j\le r} \nu_j^{-1}\|x\|_{l_{p_j}^N},
\end{align}
let $X^*$ be its dual, let $I$ be the identity operator on $\R^N$, and let
\begin{align}
\label{big_a_def} A = \min _{1\le j\le r} \nu_j s^{1/q-1/p_j}.
\end{align}
Then
\begin{align}
\label{dn_est_a_sq}
(d^n(\cap _{j=1}^r \nu_j B_{p_j}^N, \, l_q^N))^2 \ge \inf _{t\ge 0}\Bigl(\frac{A^2}{2} - 2A^2 \frac{n^{1/2}s^{1/2-1/q}}{N^{1/2}} \|I\|_{X^* \rightarrow l_2^N}t + ct^2\Bigr).
\end{align}
\end{Lem}
\begin{proof}
First we prove that
\begin{align}
\label{norm_hat_x} \|\hat x\|_{X^*} = A.
\end{align}
Indeed,
$$
\|\hat x\|_{X^*} = \max _{y\in B_X} \langle \hat x, \, y\rangle \stackrel{\eqref{hat_xj_def}, \eqref{norm_x}}{=} \max \Bigl\{ \sum \limits _{i=1}^s s^{-1/q'}y_i:\; y \in \cap _{j=1}^r \nu_j B_{p_j}^N\Bigr\} =$$$$= \min _{1\le j\le r} \nu_j s^{1/p_j'-1/q'} = A.
$$

Let $f_{\hat x}(y) = \sum \limits _{j=1}^N \alpha_jy_j$. We will choose $\alpha=(\alpha_1, \, \dots, \, \alpha_N)$ such that $f_{\hat x}$ is a support functional at the point $\hat x \in X^*$.

We have
$\|f_{\hat x}\| = \|\alpha\|_X$.
Let us prove that if $\alpha = (t, \, \dots, \, t, \, 0, \, \dots, \, 0)$, where $t>0$ is at the first $s$ positions, then
$$
\sum \limits _{j=1}^N \alpha_j\hat x_j = \|\alpha\|_X \|\hat x\|_{X^*}.
$$
Indeed, the left-hand side, by \eqref{hat_xj_def}, is equal to $st\cdot s^{-1/q'} = ts^{1/q}$; the right-hand side, by \eqref{norm_x}, \eqref{big_a_def}, \eqref{norm_hat_x}, is equal to
$$
\max _{1\le j\le N} \nu_j^{-1} \|\alpha\|_{l^N_{p_j}} \min _{1\le j\le r} \nu_j s^{1/q-1/p_j} = \max _{1\le j\le N} \nu_j^{-1} ts^{1/p_j} \min _{1\le j\le r} \nu_j s^{1/q-1/p_j} = ts^{1/q}.
$$
Hence if we choose $t$ such that $\|\alpha\|_X=1$, then $f(y)= t\sum \limits _{j=1}^s y_j$ is the support functional at $\hat x$. We see that $t \stackrel{\eqref{norm_x}}{=} \min _{1\le j\le r} \nu_js^{-1/p_j} \stackrel{\eqref{big_a_def}}{=} A\cdot s^{-1/q}$.

Therefore, $f_{\hat x}(y) =A\cdot s^{-1/q}\sum \limits _{i=1}^s y_i$ is the support functional at $\hat x$.

Similarly, for each $g\in G$, 
\begin{align}
\label{f_gx_y1}
\|g\hat x\|_{X^*} =A,
\end{align}
and if
\begin{align}
\label{f_gx_y}
f_{g\hat x}(y) = A\cdot s^{-1/q}\sum \limits _{i=1}^N g(\hat z)_iy_i,
\end{align}
where $\hat z = s^{1/q'}\hat x$ (i.e., the first $k$ coordinates of $\hat z$ are 1, and the others are 0), then $f_{g\hat x}$ is the support functional at $g\hat x$.

By Theorem \ref{dual_teor},
$$
d^n(\cap _{j=1}^r \nu_j B_{p_j}^N, \, l_q^N) = d_n(B^N_{q'}, \, X^*).
$$

Now we argue similarly as in \cite{gluskin1}. Let $L\subset \R^N$ be a subspace of dimension at most $n$, let $y_g$ be a nearest to $g\hat x$ element of $L$ with respect to the norm of $X^*$. Then
\begin{align}
\label{max_aver}
\max _{g\in G} \|g\hat x-y_g\|^2_{X^*} \ge \frac{1}{|G|}\sum \limits _{g\in G} \|g\hat x-y_g\|^2_{X^*}.
\end{align}
Applying \eqref{quad_norm}, \eqref{f_gx_y1} and \eqref{f_gx_y}, we get that the right-hand side of \eqref{max_aver} is greater or equal to
$$
\frac{A^2}{2} -\frac{1}{|G|} \sum \limits _{g\in G} 2A \sum \limits _{i=1}^N A\cdot s^{-1/q} (g\hat z)_i(y_g)_i +c\cdot \frac{1}{|G|} \sum \limits _{g\in G} \|y_g\|_{X^*}^2 =:T.
$$
In \cite{gluskin1} it was proved that
$$
\left|\frac{1}{|G|} \sum \limits _{g\in G}\sum \limits _{i=1}^N (g\hat z)_i(y_g)_i \right| \le \frac{n^{1/2}s^{1/2}}{N^{1/2}} \left(\frac{1}{|G|} \sum \limits _{g\in G}\|y_g\|^2_{l_2^N}\right)^{1/2}.
$$
Hence
$$
T\ge \frac{A^2}{2} - 2A^2 \frac{n^{1/2}s^{1/2-1/q}}{N^{1/2}} \|I\| _{X^* \rightarrow l_2^N} \left(\frac{1}{|G|} \sum \limits _{g\in G}\|y_g\|^2_{X^*}\right)^{1/2} + c \frac{1}{|G|} \sum \limits _{g\in G}\|y_g\|^2_{X^*}\ge
$$
$$
\ge \inf _{t\ge 0}\Bigl(\frac{A^2}{2} - 2A^2 \frac{n^{1/2}s^{1/2-1/q}}{N^{1/2}} \|I\|_{X^* \rightarrow l_2^N}t + ct^2\Bigr)=:T'.
$$
Therefore, by \eqref{hat_xj_def} and \eqref{max_aver}, $\sup_{x\in B^N_{q'}}\inf _{y\in L}\|x-y\|^2\ge T'$; this implies \eqref{dn_est_a_sq}.
\end{proof}

\section{Proofs of main results}

\renewcommand{\proofname}{\bf Proof of Theorem \ref{main1}}

\begin{proof}
{\it The upper estimate.} By assertion 3 of Theorem \ref{spg_teor},
$$
d^n(\cap_{j=1}^r \nu_j B_{p_j}^N, \, l_q^N)\le \min \{\nu_1 d^n(B_{p_1}^N, \, l_q^N), \, \nu_r d^n(B_{p_r}^N, \, l_q^N)\} \underset{p_1}{\lesssim} $$$$ \lesssim \min \{\nu_r, \, \nu_1n^{-1/2} N^{1/p_1'}\}.
$$

{\it The lower estimate.} If $\nu_r B_2^N \subset \cap_{j=1}^r \nu_j B_{p_j}^N$, then
$$
d^n(\cap_{j=1}^r \nu_j B_{p_j}^N, \, l_q^N) \ge d^n(\nu_r B_2^N, \, l_q^N) \gtrsim \nu_r
$$
by assertion 2 of Theorem \ref{spg_teor}.

Let
\begin{align}
\label{not_incl_b2} \nu_r B_2^N \not\subset \cap_{j=1}^r \nu_j B_{p_j}^N.
\end{align}
We prove that there is an absolute constant $a>0$ such that
\begin{align}
\label{t1_small_n}
d^n(\cap_{j=1}^r \nu_j B_{p_j}^N, \, l_q^N) \gtrsim \nu_r \quad \text{if } n^{1/2} \le a\frac{\nu_1}{\nu_r} N^{1/p_1'}
\end{align}
(notice that here $p_j$ may be greater than $q$),
\begin{align}
\label{est_big_n}
d^n(\cap_{j=1}^r \nu_j B_{p_j}^N, \, l_q^N) \gtrsim \nu_1n^{-1/2}N^{1/p_1'}\quad \text{if } a \frac{\nu_1}{\nu_r}N^{1/p_1'} \le n^{1/2}\le (N/2)^{1/2}.
\end{align}

First we show that it is sufficient to prove \eqref{t1_small_n}, \eqref{est_big_n} for $p_r\le 2$. Indeed, let
$2<p_r\le \infty$. We set $j_0=\max \{j=1, \, \dots, \, r:\; p_j < 2\}$.
If $\nu_{j_0}N^{-1/p_{j_0}} \le \nu_rN^{-1/2}$, we set $j_1=j_0+1$; otherwise, we set $j_1=\min \{j=1, \, \dots, \, j_0:\; \nu_j N^{-1/p_j} > \nu_r N^{-1/2}\}$. Then, for $j_1\le j\le j_0$, we have $\nu_j B^N_{p_j}\supset \nu_r B^N_2$, and 
$$
d^n(\cap _{j=1}^r \nu_j B^N_{p_j}, \, l_q^N)= d^n(\nu_r B^N_2\cap(\cap _{j=1}^{j_1-1} \nu_j B^N_{p_j}), \, l_q^N).
$$
By \eqref{not_incl_b2}, $j_1>1$. Now we can apply \eqref{t1_small_n} or \eqref{est_big_n} for $\nu_r B^N_2\cap(\cap _{j=1}^{j_1-1} \nu_j B^N_{p_j})$ since analogue of \eqref{ne_vkl} holds.

Let us prove \eqref{t1_small_n} for $p_r\le 2$; then $p_j\le 2$, $1\le j\le N$. We apply Lemma \ref{low_est_lem} for $s=1$. By \eqref{ne_vkl} and \eqref{big_a_def}, we have $A=\nu_r$. Hence
$$
(d^n(\cap_{j=1}^r \nu_j B_{p_j}^N, \, l_q^N))^2 \stackrel{\eqref{dn_est_a_sq}}{\ge} \inf _{t\ge 0} \Bigl( \frac{\nu_r^2}{2} - 2\nu_r^2 \frac{n^{1/2}}{N^{1/2}} \|I\|_{X^* \rightarrow l_2^N} t + ct^2\Bigr),
$$
where $X^*$ is the dual space of $(\R^N, \, \max _{1\le j\le r} \nu_j^{-1}\|\cdot\|_{l_{p_j}^N})$.

We have $\|I\|_{X^* \rightarrow l_2^N} = \|I\|_{l_2^N \rightarrow X}$; i.e., it is the smallest constant $C>0$ in the inequality $\max _{1\le j\le r} \nu_j^{-1}\|x\|_{l_{p_j}^N}\le C\|x\|_{l_2^N}$. By H\"{o}lder's inequality, it is not greater than $\max _{1\le j\le r} \nu_j^{-1}N^{1/p_j-1/2} \stackrel{\eqref{ne_vkl}}{=} \nu_1^{-1} N^{1/p_1-1/2}$. Therefore,
$$
\frac{\nu_r^2}{2} - 2\nu_r^2 \frac{n^{1/2}}{N^{1/2}} \|I\|_{X^* \rightarrow l_2^N} t + ct^2 \ge \frac{\nu_r^2}{2} - 2\nu_r^2 \frac{n^{1/2}}{\nu_1N^{1/p_1'}} t + ct^2=:T(t).
$$
Since $n^{1/2} \le a \frac{\nu_1}{\nu_r}N^{1/p_1'}$, we have
$$
\inf _{t\ge 0}T(t)\ge \inf _{t\ge 0}\Bigl(\frac{\nu_r^2}{2} - 2\nu_r at+ct^2\Bigr) \ge \frac{\nu_r^2}{4}
$$
for a sufficiently small absolute constant $a>0$.

This completes the proof of \eqref{t1_small_n} for $p_r\le 2$ and, consequently, for all $p_j\in [1, \, +\infty]$, $1\le j\le r$.

Now we prove \eqref{est_big_n} for $p_r\le 2$.
Notice that it suffices to consider the case $a\frac{\nu_1}{\nu_r} > \frac{1}{\sqrt{2}}$ (otherwise, by \eqref{ne_vkl}, the problem can be reduced to estimating $d^n(\nu_1B_{p_1}^N, \, l_q^N)$ and applying assertion 3 of Theorem \ref{spg_teor}; by the remark after the formulation of this theorem, the constant in the lower estimate is absolute).

We have $$a \frac{\nu_1}{\nu_r}N^{1/p_1'} \le n^{1/2}< a \frac{\nu_1}{\nu_r}N^{1/2}.$$ We define the number $\tilde p_1$ by the equation 
\begin{align}
\label{til_p_def}
n^{1/2} = a \frac{\nu_1}{\nu_r}N^{1/\tilde p_1'}. 
\end{align}
Then $1/\tilde p_1'\in [1/p_1', \, 1/2)$. Recall that $1/p_j'\in (0, \, 1/2]$ since $1<p_j\le 2$ $(j=1, \, \dots, \, r)$. For $2\le j\le r$, we define the numbers $\tilde p_j$ by the equations
\begin{align}
\label{1pjtil}
1/\tilde p_j' = 1/p_j'+1/\tilde p_1'-1/p_1' \in (0, \, 1].
\end{align}
Then
$$
d^n(\cap_{j=1}^r \nu_j B_{p_j}^N, \, l_q^N) \ge N^{1/p_1'-1/\tilde p_1'}d^n(\cap _{j=1}^r \nu_j B^N_{\tilde p_j}, \, l_q^N) \stackrel{\eqref{til_p_def}}{=} $$$$=a \frac{\nu_1}{\nu_r} n^{-1/2}N^{1/p_1'} d^n(\cap _{j=1}^r \nu_j B^N_{\tilde p_j}, \, l_q^N).
$$
Therefore, in order to prove \eqref{est_big_n} it suffices to check that 
$$
d^n(\cap _{j=1}^r \nu_j B^N_{\tilde p_j}, \, l_q^N) \gtrsim \nu_r.
$$
It follows from \eqref{til_p_def} and \eqref{t1_small_n} (here we use that $N^{1/\tilde p_i-1/\tilde p_j} \stackrel{\eqref{1pjtil}}{=} N^{1/p_i-1/p_j}$, $j=1, \, \dots, \, r$, and the analogue of \eqref{ne_vkl} holds for $\{\tilde p_i\}_{i=1}^r$).
\end{proof}

\renewcommand{\proofname}{\bf Proof of Theorem \ref{main2}}

\begin{proof}
The upper estimate follows from the inclusion $\cap _{k=1}^r \nu_k B_{p_k}^{N} \subset \nu_i^{1-\lambda_{ij}} \nu_j^{\lambda_{ij}} B_q^N$, where $p_i\le q$, $p_j\ge q$, $\lambda_{ij}\in [0, \, 1]$ satisfies \eqref{lambda_def} (see Theorem \ref{incl_teor}).

Let us prove the lower estimate. It suffices to consider the case $p_i\ne q$, $1\le i\le r$ (in the general case, the desired estimate can be obtained by passing to limit). Let $p_{i_*}<q<p_{j_*}$,
\begin{align}
\label{min_ij}
\nu_{i_*}^{1-\lambda_{i_*j_*}} \nu_{j_*} ^{\lambda_{i_*j_*}} = \min _{p_i< q, \, p_j> q} \nu_i^{1-\lambda_{ij}} \nu_j^{\lambda_{ij}},
\end{align}
where $\lambda_{ij}$ are defined by \eqref{lambda_def}.

It suffices to prove that, for $N=2^m$, $n\le N/2$,
\begin{align}
\label{2m_est} d^n(\cap _{j=1}^r \nu_j B_{p_j}^{N}, \, l_q^N) \gtrsim \nu_{i_*}^{1-\lambda_{i_*j_*}} \nu_{j_*} ^{\lambda_{i_*j_*}}.
\end{align}
We define the number $s$ by the equation
\begin{align}
\label{s_def} \frac{\nu_{i_*}}{\nu_{j_*}} = s^{1/p_{i_*}-1/p_{j_*}}.
\end{align}
By \eqref{ne_vkl}, $1\le s\le N$. Let $k\in \{0, \, \dots, \, m\}$ be such that 
\begin{align}
\label{k_s_def}
2^k\le s<2^{k+1}. 
\end{align}
We have $\|x\|_{l^m_q} \ge 2^{-k/q'}\|x\|_{l^{2^{m-k}, 2^k}_{\infty,1}}$. Let $a = \min_{1\le j\le r} \nu_j \cdot 2^{-k/p_j}$. Then from the condition $p_j\ge 2$, $1\le j\le r$, it follows that $a\cdot B_{2,\infty}^{2^{m-k},2^k} \subset \cap _{j=1}^r \nu_j B_{p_j}^{2^m}$. Hence
$$
d^n(\cap _{j=1}^r \nu_j B_{p_j}^{N}, \, l_q^N) \ge 2^{-k/q'}a\cdot d^n(B_{2,\infty}^{2^{m-k},2^k}, \, l^{2^{m-k},2^k}_{\infty,1}) \stackrel{\eqref{mr_th},\eqref{k_s_def}}{\gtrsim} $$$$ \gtrsim s\cdot s^{-1/q'}\min_{1\le j\le r} \nu_j s^{-1/p_j} = s^{1/q}\min_{1\le j\le r} \nu_j s^{-1/p_j}.
$$
If
\begin{align}
\label{a_calc} \min_{1\le j\le r} \nu_j s^{-1/p_j} =\nu_{j_*} s^{-1/p_{j_*}},
\end{align}
the right-hand side is equal to
$$
\nu_{j_*} s^{1/q-1/p_{j_*}} \stackrel{\eqref{lambda_def}}{=} \nu_{j_*} s^{(1-\lambda_{i_*j_*})(1/p_{i_*}-1/p_{j_*})} \stackrel{\eqref{s_def}}{=} \nu_{i_*}^{1-\lambda_{i_*j_*}} \nu_{j_*}^{\lambda_{i_*j_*}},
$$
which yields the desired lower estimate.

Let us check \eqref{a_calc}. First we notice that $\nu_{j_*} s^{-1/p_{j_*}} = \nu_{i_*} s^{-1/p_{i_*}}$ by \eqref{s_def}. 

Let $p_j> q$. We define the number $\tilde s$ by the equation 
\begin{align}
\label{til_s_def}
\frac{\nu_j}{\nu_{i_*}} = \tilde s^{1/p_j-1/p_{i_*}}.
\end{align}
We show that $\nu_{i_*}s^{1/p_j-1/p_{i_*}}\le \nu_j$. It is equivalent to the inequality $s^{1/p_j-1/p_{i_*}} \le \tilde s^{1/p_j-1/p_{i_*}}$, or
\begin{align}
\label{sgs}
s\ge \tilde s
\end{align}
(since $p_j>q>p_{i_*}$).

By \eqref{min_ij}, $\nu_{i_*}^{1-\lambda_{i_*j_*}}\nu_{j_*}^{\lambda_{i_*j_*}} \le \nu_{i_*}^{1-\lambda_{i_*j}}\nu_{j}^{\lambda_{i_*j}}$, or $s^{\lambda_{i_*j_*}(1/p_{j_*}-1/p_{i_*})} \le \tilde s^{\lambda_{i_*j}(1/p_{j}-1/p_{i_*})}$ (see \eqref{s_def}, \eqref{til_s_def}). By \eqref{lambda_def}, we get $s^{1/q-1/p_{i_*}} \le \tilde s^{1/q-1/p_{i_*}}$, which implies \eqref{sgs}.

Let $p_i<q$. We will show that $\nu_{j_*}s^{1/p_i-1/p_{j_*}}\le \nu_i$. Now we define the number $\tilde s$ by the equation $\frac{\nu_i}{\nu_{j_*}} = \tilde s^{1/p_i-1/p_{j_*}}$ and check the inequality $\nu_{j_*}s^{1/p_i-1/p_{j_*}}\le \nu_{j_*}\tilde s^{1/p_i-1/p_{j_*}}$, or $s\le \tilde s$. As in the previous case, the last relation can be derived from the inequality $\nu_{i_*}^{1-\lambda_{i_*j_*}}\nu_{j_*}^{\lambda_{i_*j_*}} \le \nu_i^{1-\lambda_{ij_*}}\nu_{j_*}^{\lambda_{ij_*}}$.
\end{proof}

\renewcommand{\proofname}{\bf Proof of Theorem \ref{cor1}}

\begin{proof}
The upper estimates for linear and Gelfand widths follow from Theorems \ref{spg_teor} and \ref{incl_teor}. The lower estimate for Gelfand widths of the intersection of the finite family of balls can be proved similarly as in \cite[Proposition 1]{vas_int_sob}; here we apply Theorems \ref{main1} and \ref{main2}. For the intersection of an arbitrary number of balls, we argue as in \cite[\S 5]{vas_mix_sev}. Linear widths can be estimated from below by the Gelfand widths.
\end{proof}

\renewcommand{\proofname}{\bf Proof of Theorem \ref{main3}}

\begin{proof}
The upper estimate follows from the inclusion $\nu_1B^N_{p_1} \cap \nu_2B^N_{p_2} \subset \nu_1^{1-\lambda} \nu_2^\lambda B_2^N$.

Let us prove the lower estimate. We show that, in the cases a) $\frac{\nu_1}{\nu_2}\le N^{1/p_1-1/2}$, $n^{1/2}\le a\left(\frac{\nu_1}{\nu_2}\right)^\lambda N^{1/p_1'}$, b) $\frac{\nu_1}{\nu_2}\ge N^{1/p_1-1/2}$, $n^{1/2}\le a\left(\frac{\nu_1}{\nu_2}\right)^{\lambda-1} N^{1/2}$ (where $a\in (0, \, 1)$ is an absolute constant), we have 
\begin{align}
\label{dn_est_sm_n} d^n(\nu_1B^N_{p_1} \cap \nu_2B^N_{p_2}, \, l_2^N) \gtrsim \nu_1^{1-\lambda} \nu_2^\lambda.
\end{align}

We define the number $\tilde s$ by the equation 
\begin{align}
\label{t4_til_s_def}
\frac{\nu_1}{\nu_2} = \tilde s^{1/p_1-1/p_2}
\end{align}
and set $s = \lceil \tilde s\rceil$. By the conditions of Theorem \ref{main3}, we have $1\le \frac{\nu_1}{\nu_2}\le N^{1/p_1-1/p_2}$; hence $s\in \{1, \, \dots, \, N\}$. We apply Lemma \ref{low_est_lem} and get
\begin{align}
\label{dn_nu1_l_e}
(d^n(\nu_1 B_{p_1}^N \cap \nu_2 B_{p_2}^N, \, l_2^N))^2 \ge \inf _{t\ge 0}\Bigl(\frac{A^2}{2} - 2A^2 \frac{n^{1/2}}{N^{1/2}} \|I\|_{X^* \rightarrow l_2^N}t + ct^2\Bigr),
\end{align}
where
\begin{align}
\label{a_t4}
A = \min \{\nu_1 s^{1/2-1/p_1}, \, \nu_2 s^{1/2-1/p_2}\} \stackrel{\eqref{t4_til_s_def}}{\asymp} \nu_1 \tilde s^{1/2-1/p_1}\stackrel{\eqref{t4_til_s_def}}{=} \nu_2 \tilde s^{1/2-1/p_2},
\end{align}
and $c$ is from Lemma \ref{qnorm}. For $\frac{\nu_1}{\nu_2} \le N^{1/p_1-1/2}$, we get
\begin{align}
\label{ixst2n}
\|I\|_{X^*\rightarrow l_2^N}= \|I\|_{l_2^N\to X} \stackrel{\eqref{norm_x}}{\le} \max \{\nu_2^{-1}, \, \nu_1^{-1}N^{1/p_1-1/2}\} = \nu_1^{-1}N^{1/p_1-1/2}.
\end{align}
Hence, if $n^{1/2}\le a\left(\frac{\nu_1}{\nu_2}\right)^\lambda N^{1/p_1'}$, then for some absolute constant $c_1>0$ we have
$$
A\frac{n^{1/2}}{N^{1/2}} \|I\|_{X^* \rightarrow l_2^N} \stackrel{\eqref{a_t4}}{\le} c_1\nu_1 \tilde s^{1/2-1/p_1}n^{1/2}N^{-1/2} \nu_1^{-1}N^{1/p_1-1/2} = $$$$\le c_1 \tilde s^{1/2-1/p_1} n^{1/2}N^{-1/p_1'} \stackrel{\eqref{121lp1}}{\le} c_1 \tilde s^{\lambda(1/p_2-1/p_1)}a \left(\frac{\nu_1}{\nu_2}\right)^\lambda \stackrel{\eqref{t4_til_s_def}}{=} c_1a.
$$
If $a$ is sufficiently small, we have 
\begin{align}
\label{dncap_nu_lam}
(d^n(\cap _{j=1}^r \nu_j B_{p_j}^N, \, l_2^N))^2 \stackrel{\eqref{dn_nu1_l_e}}{\ge} \frac{A^2}{4} \stackrel{\eqref{a_t4}}{\asymp} \nu_1^2 \tilde s^{1-2/p_1} \stackrel{\eqref{121lp1}, \eqref{t4_til_s_def}}{=} (\nu_1^{1-\lambda} \nu_2^\lambda)^2;
\end{align}
this completes the proof of \eqref{dn_est_sm_n} in case a).

In case b) we have
$$
\|I\|_{X^*\rightarrow l_2^N} \le \max \{\nu_2^{-1}, \, \nu_1^{-1}N^{1/p_1-1/2}\} = \nu_2^{-1}.
$$
Hence, if $n^{1/2}\le a\left(\frac{\nu_1}{\nu_2}\right)^{\lambda-1} N^{1/2}$, then, for some absolute constant $c_2>0$,
$$
A\frac{n^{1/2}}{N^{1/2}} \|I\|_{X^* \rightarrow l_2^N} \stackrel{\eqref{a_t4}}{\le} c_2\nu_2 \tilde s^{1/2-1/p_2}n^{1/2}N^{-1/2} \nu_2^{-1} \le $$
$$
\le c_2 \tilde s^{1/2-1/p_2}a\left(\frac{\nu_1}{\nu_2}\right)^{\lambda-1} \stackrel{\eqref{121lp1},\eqref{t4_til_s_def}}{=} c_2a;
$$
if $a$ is sufficiently small, then from \eqref{dn_nu1_l_e} we get \eqref{dncap_nu_lam}; hence \eqref{dn_est_sm_n} holds.

This completes the proof of assertion 2 of Theorem \ref{main3}, as well as assertion 1 for $n^{1/2}\le a\left(\frac{\nu_1}{\nu_2}\right)^\lambda N^{1/p_1'}$. 

Let us prove the estimate in assertion 1 for 
\begin{align}
\label{p2_usl}
a\left(\frac{\nu_1}{\nu_2}\right)^\lambda N^{1/p_1'}<n^{1/2}\le aN^{1/2}.
\end{align}

We set
\begin{align}
\label{tils_2} \tilde s =\Bigl(a^{-1}n^{1/2}N^{-1/p_1'}\Bigr) ^{\frac{1}{1/p_1-1/2}}, \quad s = \lceil \tilde s\rceil.
\end{align}
From \eqref{p2_usl} it follows that $1\le s\le N$.

We again obtain \eqref{dn_nu1_l_e}, where
\begin{align}
\label{a_2_por}
A = \min \{\nu_1 s^{1/2-1/p_1}, \, \nu_2 s^{1/2-1/p_2}\} \stackrel{\eqref{121lp1}, \eqref{p2_usl}, \eqref{tils_2}}{\asymp} \nu_1a\cdot n^{-1/2} N^{1/p_1'},
\end{align}
$$
\|I\|_{X^*\rightarrow l_2^N} \stackrel{\eqref{ixst2n}}{\le} \nu_1^{-1}N^{1/p_1-1/2}.
$$
Hence
$$
A\frac{n^{1/2}}{N^{1/2}}\|I\|_{X^*\rightarrow l_2^N}\lesssim a \cdot \nu_1 n^{-1/2} N^{1/p_1'} n^{1/2}N^{-1/2} \nu_1^{-1}N^{1/p_1-1/2} =a.
$$
Therefore, for some absolute constant $b>0$, we get
$$
\inf_{t\ge 0}\Bigl(\frac{A^2}{2} - 2A^2 \frac{n^{1/2}}{N^{1/2}} \|I\|_{X^* \rightarrow l_2^N}t + ct^2\Bigr) \ge \inf_{t\ge 0}\Bigl(\frac{A^2}{2} - 2A\cdot abt+ct^2\Bigr)\ge \frac{A^2}{4}
$$
if $a>0$ is sufficiently small. Now we apply \eqref{dn_nu1_l_e}, \eqref{a_2_por}, and get $$d^n(\nu_1B^N_{p_1}\cap \nu_2B^N_{p_2}, \, l_q^N) \gtrsim \nu_1 n^{-1/2}N^{1/p_1'}.$$
This completes the proof.
\end{proof}

\section{Application: estimates for the widths of intersections of Sobolev classes}

We consider the example from \cite{vas_int_sob} of the intersection of Sobolev classes on a John domain; the example of the intersection of one-dimensional periodic Sobolev classes (where the smoothness may be noninteger) can be considered similarly.

Let $\Omega \subset \R^d$ be a bounded domain, $1\le p\le \infty$, $r\in \Z_+$. Given $f\in L_1^{{\rm loc}}(\Omega)$, we denote by $\nabla^r f$ the vector of all partial derivatives of order $r$ (in the sense of distributions). If all its components are elements of $L_p(\Omega)$, we say that $f$ is an element of the Sobolev space ${\cal W}^r_p(\Omega)$. The set
$$
W^r_p(\Omega) = \{f\in {\cal W}^r_p(\Omega):\; \|\nabla^r f\|_{L_p(\Omega)}\le 1\}
$$
is the Sobolev class; here $\|\nabla^r f\|_{L_p(\Omega)}$ is the $L_p$-norm of the function $|\nabla^r f(\cdot)|$.

We recall the definition of the John domain.

Let $B_a(x)$ be the Euclidean ball of radius $a$ centered at point $x$.

\begin{Def}
\label{fca} Let $\Omega\subset\R^d$ be a bounded domain, and let $a>0$. We say that $\Omega \in {\bf FC}(a)$ if there is a point $x_*\in \Omega$
such that, for each $x\in \Omega$, there is a number $T(x)>0$ and a curve $\gamma _x:[0, \, T(x)] \rightarrow\Omega$ with the following properties:
\begin{enumerate}
\item $\gamma _x$ has the natural parametrization with respect to the Euclidean norm on $\R^d$,
\item $\gamma _x(0)=x$, $\gamma _x(T(x))=x_*$,
\item $B_{at}(\gamma _x(t))\subset \Omega$ for all $t\in [0, \, T(x)]$.
\end{enumerate}
We say that $\Omega$ satisfies the John condition if $\Omega\in {\bf FC}(a)$ for some $a>0$.
\end{Def}

Let $a>0$, $\Omega \in {\bf FC}(a)$ be a domain in $\R^d$, $s\ge 2$, $r_i\in \Z_+$, $1\le i\le s$, 
\begin{align}
\label{r_order}
r_1>r_2>\dots>r_s, 
\end{align}
$1<p_i\le \infty$, $1\le i\le s$. We set
\begin{align}
\label{m_def} M = \cap _{j=1}^s W^{r_j}_{p_j}(\Omega).
\end{align}

Let
\begin{align}
\label{r_i_p_i} \frac{r_j}{d}-\frac{1}{p_j} < \frac{r_i}{d} -\frac{1}{p_i} \quad \text{for}\; j<i.
\end{align}
This implies that $p_j<p_i$ for $j<i$. The general case when \eqref{r_i_p_i} may fail can be reduced to this particular case similarly as in \cite[\S 3]{vas_int_sob}.

\renewcommand{\proofname}{\bf Proof}

\begin{Trm}
Let $1\le q<\infty$, $s\ge 2$, $1<p_j\le \infty$, $j=1, \, \dots, \, s$, $d\in \N$, $a>0$. Let $\Omega\in {\bf FC}(a)$ be a domain in $\R^d$, let the set $M$ be defined by \eqref{m_def}. Suppose that \eqref{r_order} and \eqref{r_i_p_i} hold, as well as the inequality $\frac{r_s}{d} + \frac 1q-\frac{1}{p_s}>0$. We write $\overline{p} = (p_1, \, \dots, \, p_s)$, $\overline{r} = (r_1, \, \dots, \, r_s)$.
\begin{enumerate}
\item Let $p_j\ge q$, $1\le j\le s$. Then
$$
d^n(M, \, L_q(\Omega)) \underset{\overline{p},q,\overline{r},d,a}{\asymp} n^{-r_1/d}.
$$

\item Let $2\le p_j\le q$, $1\le j\le s$. Then 
$$
d^n(M, \, L_q(\Omega)) \underset{\overline{p},q,\overline{r},d,a}{\asymp} n^{-r_s/d-1/q+1/p_s}.
$$

\item Let $q\ge 2$, $p_j\le q$, $1\le j\le s$, and let $p_1<2$.
\begin{enumerate}
\item If $\frac{r_1-r_s}{d}\le \frac 12-\frac{1}{p_s}$, then
$$
d^n(M, \, L_q(\Omega)) \underset{\overline{p},q,\overline{r},d,a}{\asymp} n^{-r_s/d-1/q+1/p_s}.
$$
\item Let $\frac{r_1-r_s}{d}> \frac 12-\frac{1}{p_s}$ and
$$\theta_1:=\frac{r_1}{d}+\frac 1q -\frac 12 \ne \frac{r_s/d+1/q-1/p_s}{2(r_s/d-r_1/d+1/p_s')}=:\theta_2.$$ Then
$$
d^n(M, \, L_q(\Omega)) \underset{\overline{p},q,\overline{r},d,a}{\asymp} n^{-\min \{\theta_1, \, \theta_2\}}.
$$
\end{enumerate}
\item Let $q\ge 2$, $p_j\ge 2$, $1\le j\le s$, and let $\{j:\; p_j>q\}\ne \varnothing$, $\{j:\; p_j<q\}\ne \varnothing$. For $i$, $j$ such that $p_i<q$, $p_j>q$ we define the numbers $\lambda_{ij}$ by the equations $\frac{1}{q} = \frac{1-\lambda_{ij}}{p_i} + \frac{\lambda_{ij}}{p_j}$, $(i_*, \, j_*)= {\rm argmax}\, _{p_i<q, \, p_j>q} ((1-\lambda_{ij})r_i/d + \lambda_{ij} r_j/d)$. Then
$$
d^n(M, \, L_q(\Omega)) \underset{\overline{p},q,\overline{r},d,a}{\asymp} n^{-(1-\lambda_{i_*j_*})r_{i_*}/d - \lambda_{i_*j_*} r_{j_*}/d}.
$$

\item Let $q=2$, $s=2$, $p_1<2<p_2$, $\frac{r_1-r_2}{d}\ge \frac 12-\frac{1}{p_2}$. We define the number $\lambda$ by the equation $\frac 12=\frac{1-\lambda}{p_1} + \frac{\lambda}{p_2}$. Suppose that
$$
\theta_1:= \frac{r_1}{d} \ne \frac{(1-\lambda)r_1+\lambda r_2}{2\lambda(r_2-r_1)+d}=:\theta_2.
$$
Then
$$
d^n(M, \, L_q(\Omega)) \underset{\overline{p},q,\overline{r},d,a}{\asymp} n^{-\min \{\theta_1, \, \theta_2\}}.
$$
\end{enumerate}
\end{Trm}
\begin{proof}
We argue similarly as in \cite{vas_int_sob}, replacing the Kolmogorov widths by the Gelfand widths. Assertions 1, 2 and 4 are similar to assertions 1, 2 and 4 of Theorem 1 from \cite{vas_int_sob}. Assertion 3(a) is also similar to \cite[Theorem 1, assertion 2]{vas_int_sob}; indeed, by Theorem \ref{main1} for $n\le 2^{m-1}$ we have
$$
d^n\Bigl(\cap _{j=1}^s 2^{-m(r_j/d+1/q-1/p_j)}B^{2^m}_{p_j}, \, l_q^{2^m}\Bigr) \underset{p_1}{\asymp} $$$$\asymp\min \{2^{-m(r_1/d+1/q-1/p_1)}n^{-1/2}2^{m/p_1'}, \, 2^{-m(r_s/d+1/q-1/p_s)}\} = 2^{-m(r_s/d+1/q-1/p_s)}.
$$
In assertions 3(b) and 5 we argue as in the proof of part 5 of Theorem 1 from \cite{vas_int_sob}; here we apply, correspondingly, Theorems \ref{main1} and \ref{main3}. The problem is reduced to searching the minimum point $t_*$ of the function $h:[1, \, T_*] \rightarrow \R$, where $h(t) = t\left( \frac{r_1}{d}+\frac 1q-1\right) + \frac 12$,
$$T_*=\begin{cases}\frac{1}{2(r_s/d-r_1/d+1/p_s')} & \text{in part 3(b)},\\
\frac{1}{2\lambda(r_2/d-r_1/d)+1} & \text{in part 5}.\end{cases}
$$
Since $\theta_1\ne \theta_2$, the minimum point of $h$ is unique and $$d^n(M, \, L_q(\Omega)) \underset{\overline{p},q,\overline{r},d,a}{\asymp} n^{-h(t_*)} = n^{-\min \{\theta_1, \, \theta_2\}}.$$
This completes the proof.
\end{proof}

\section{Some generalizations of Lemma \ref{low_est_lem}}

Let $H \subset S_N$ be a subgroup acting transitively on $\{1, \, \dots, \, N\}$ (i.e., for all $i$, $j\in \{1, \, \dots, \, N\}$ there is a permutation $h\in H$ such that $h(i)=j$), $G = H\times \{-1, \, 1\}^N$, an element $g\in G$ is acting on $x=(x_1, \, \dots, \, x_N)\in \R^N$ according to formula \eqref{dx_dejstv}. In \cite{mal_rjut1} order estimates for the Kolmogorov widths of sets invariant with respect to $G$ in the space $l_q^N$ for $1\le q\le 2$ were obtained.

Let $\|\cdot \|_X$ be a norm on $\R^N$ invariant with respect to $G$; i.e., $\|gx\|_X = \|x\|_X$, $x\in X$, $g\in G$. By $\langle\cdot, \, \cdot\rangle$ we denote the standard inner product on $\R^N$.

Let $\hat x\ne 0$, $V = {\rm conv}\, \{g\hat x\}_{g\in G}$. Let $b=(b_1, \, \dots, \, b_N)\in \R^N$ be such that $x \stackrel{x^*_b}{\mapsto} \langle b, \, x\rangle$ is the supporting functional at $\hat x$; i.e., $\|x^*_b\|_{X^*}=1$ and $\langle b, \, \hat x\rangle = \|\hat x\|_X$. Then $x \mapsto \langle gb, \, x\rangle$ is the supporting functional at $g\hat x$.

\begin{Sta}
Let $0\le n\le N$. Then
\begin{align}
\label{dnv2}
(d_n(V, \, X))^2 \ge \inf _{t\ge 0}\Bigl( \frac{\|\hat x\|^2_X}{2} -2\|\hat x\|_X \cdot \frac{\|b\|_{l_2^N} n^{1/2}}{N^{1/2}}\|I\|_{X \rightarrow l_2^N}t +ct^2\Bigr),
\end{align}
where $c>0$ is the absolute constant from Lemma {\rm \ref{qnorm}}.
\end{Sta}
\begin{proof}
Let $L\subset \R^N$ be a subspace of dimension at most $n$, and let $y_g$ be the nearest to $g\hat x$ element of $L$ with respect to $\|\cdot\|_X$.

As in Lemma \ref{low_est_lem}, we get from \eqref{quad_norm} that
\begin{align}
\label{dnvx2_low}
(d_n(V, \, X))^2 \ge \frac{\|\hat x\|^2_X}{2} -2\|\hat x\|_X  \cdot \frac{1}{|G|} \sum \limits _{g\in G} \langle gb, \, y_g \rangle + \frac{c}{|G|} \sum \limits _{g\in G}\|y_g\|^2_X.
\end{align}

As in \cite{gluskin1}, we consider the space $L_2(G)$ with the inner product $\langle \varphi, \, \psi \rangle_{L_2(G)} = \frac{1}{|G|} \sum \limits _{g\in G} \varphi(g)\psi(g)$, and set $\varphi_i(g) = (gb)_i$, $z_i(g)=(y_g)_i$, $i=1, \, \dots, \, N$. Let $Z = {\rm span}\, \{z_i\}_{i=1}^N \subset L_2(G)$, and let $P$ be the orthogonal projection onto $Z$. Then $\dim Z\le n$. Similar to \cite{gluskin1}, we get
$$
\Bigl| \frac{1}{|G|} \sum \limits _{g\in G} \langle gb, \, y_g \rangle\Bigr| \le \Bigl(\sum \limits _{i=1}^N \|P\varphi_i\|^2\Bigr)^{1/2} \Bigl(\sum \limits _{i=1}^N \|z_i\|^2\Bigr)^{1/2}.
$$
Notice that for $i\ne j$ we have $\langle \varphi_i, \, \varphi_j\rangle_{L_2(G)} = 0$, and for $i=j$,
$$
\langle \varphi_i, \, \varphi_i\rangle_{L_2(G)} = \frac{1}{|G|} \sum \limits _{g\in G} (gb)_i^2 = \frac{1}{|H|} \sum \limits _{h\in H} b_{h(i)}^2 = \frac{\|b\|^2_{l_2^N}}{N}
$$
(the last equality holds since $H$ acts transitively on $\{1, \, \dots, \, N\}$). Since $\dim Z\le n$, the Hilbert--Schmidt norm of $P$ is at most $n^{1/2}$. Hence
$$
\Bigl(\sum \limits _{i=1}^N \|P\varphi_i\|^2\Bigr)^{1/2} \le \frac{n^{1/2}\|b\|_{l_2^N}}{N^{1/2}}.
$$
Further,
$$
\Bigl(\sum \limits _{i=1}^N \|z_i\|^2\Bigr)^{1/2} = \Bigl(\frac{1}{|G|}\sum \limits _{g\in G} \|y_g\|^2_{l_2^N}\Bigr)^{1/2} \le \|I\|_{X \rightarrow l_2^N}\Bigl(\frac{1}{|G|}\sum \limits _{g\in G} \|y_g\|^2_X\Bigr)^{1/2}.
$$
This together with \eqref{dnvx2_low} yields \eqref{dnv2}.
\end{proof}

\begin{Cor}
Let $\|\cdot\|_Y$ be a norm on $\R^N$ invariant with respect to $G$, and let $\hat x\in \R^N$, $\|\hat x\|_Y=1$. Then there is an absolute constant $\alpha>0$ such that, for $n \le \alpha \frac{N}{\|b\|^2_{l_2^N}\|I\|^2_{X \rightarrow l_2^N}}$,
$$
d^n(B_{X^*}, \, Y^*)=d_n(B_Y, \, X) \ge \frac{\|\hat x\|_X}{4}.
$$
\end{Cor}

\end{document}